\documentclass[reqno, 11pt]{amsart} 

\usepackage{amsmath, amssymb} 
\usepackage{amscd}            
\usepackage{amsxtra}          
\usepackage{upref}            
\usepackage{mathdots}
\usepackage{comment}

\usepackage[mathscr]{eucal}

\theoremstyle{plain}

\newtheorem{theorem}{Theorem}
\newtheorem{proposition}[theorem]{Proposition}

\newtheorem{corollary}[theorem]{Corollary}

\theoremstyle{remark}
\newtheorem{remark}[theorem]{Remark}

\numberwithin{equation}{section}
\numberwithin{theorem}{section}
\newcommand{\be}%
  {\protect\setcounter{equation}{\value{subsubsection}}}  
\newcommand{\ee}%
  {\protect\setcounter{subsubsection}{\value{equation}}}

\DeclareMathAlphabet\BOONDOX{U}{rsfso}{m}{n}

\newcommand{\N}{\mathbb{N}}
\newcommand{\Z}{\mathbb{Z}}
\newcommand{\Q}{\mathbb{Q}}
\newcommand{\R}{\mathbb{R}}
\newcommand{\C}{\mathbb{C}} 
\newcommand{\uh}{\mathcal{H}}

\newcommand{\of}{{\mathcal O}_F}

\newcommand{\ov}{{\mathcal O}_v}

\newcommand{\ak}{\mathbb{A}_{k}}

\newcommand{\sfnf}{{\mathcal S}(F^n)}

\newcommand{\hcsfmu}{{\mathcal C}_{\mu}(N_r\backslash G_r,\psi)}

\newcommand{\matn}{{\mathfrak{gl}}_n}
\newcommand{\lieb}{\mathfrak{b}}
\newcommand{\liebn}{{\mathfrak{b}}_n}

\newcommand{\gabygk}{{\rm G}_r(k)\backslash{\rm G}_r(\ak)}

\newcommand{\glr}{{\rm GL}_r}

\newcommand{\gln}{{\rm GL}_n}

\newcommand{\cs}{\mathcal S}

\newcommand{\wpif}{{\mathcal W}(\pi,\psi)}
\newcommand{\wpiflambda}{{\mathcal W}(\pi_{\lambda},\psi)}
\newcommand{\wpiflambdao}{{\mathcal W}(\pi_{\lambda_0},\psi)}

\newcommand{\wPi}{{\mathcal W}(\Pi,\psi)}

\newcommand{\wpidualf}{{\mathcal W}(\tilde{\pi},\overline{\psi})}

\newcommand{\res}{\text{Re}(s)}
\newcommand{\re}[1]{\text{Re}(#1)}

\newcommand{\lepif}{L(s,\pi,\wedge^2)}
\newcommand{\lgepif}{L_G(s,\pi,\wedge^2)}
\newcommand{\lgepiv}{L_G(s,\pi_v,\wedge^2)}
\newcommand{\epgepif}{\varepsilon_G(s,\pi,\wedge^2,\psi)}

\newcommand{\lePiv}{L(s,\Pi,\wedge^2)}

\newcommand{\ljsepi}{L_{JS}(s,\pi,\wedge^2)}
\newcommand{\ljsepif}{L_{JS}(s,\pi,\wedge^2)}
\newcommand{\ljsepiv}{L_{JS}(s,\pi_v,\wedge^2)}

\newcommand{\ljsePi}{L_{JS}(s,\Pi,\wedge^2)}

\newcommand{\epjsepif}{\varepsilon_{JS}(s,\pi,\wedge^2,\psi)}
\newcommand{\epjsepiv}{\varepsilon_{JS}(s,\pi_v,\wedge^2,\psi_v)}

\newcommand{\jswphif}{J(s,W,\phi)}
\newcommand{\jswvphiv}{J(s,W_v,\phi_v)}

\newcommand{\jswf}{J(s,W)}

\newcommand{\jswv}{J(s,W_{1,v})}

\newcommand{\jsw}{J(s,W)}

\newcommand{\lshepiv}{L_{Sh}(s,\pi_v,\wedge^2)}

\newcommand{\lshepif}{L_{Sh}(s,\pi,\wedge^2)}

\newcommand{\lshePi}{L_{Sh}(s,\Pi,\wedge^2)}
\newcommand{\lshePiv}{L_{Sh}(s,\Pi_v,\wedge^2)}

\newcommand{\epshepiv}{\varepsilon_{Sh}(s,\pi_v,\wedge^2,\psi_v)}

\newcommand{\epshepif}{\varepsilon_{Sh}(s,\pi,\wedge^2,\psi)}

\newcommand{\epshePi}{\varepsilon_{Sh}(s,\Pi,\wedge^2)}
\newcommand{\epshePiv}{\varepsilon_{Sh}(s,\Pi_v,\wedge^2,\psi_v)}

\newcommand{\epgepiv}{\varepsilon_G(s,\pi_v,\wedge^2,\psi_v)}

\DeclareMathOperator{\tr}{Tr}
\DeclareMathOperator{\Ind}{\mathit{i}}
\DeclareMathOperator{\Irr}{Irr}
\DeclareMathOperator{\Hom}{Hom}

\newcommand{\irredrepr}{\Irr(G_r)}
\newcommand{\grepfr}{\Irr_{\rm{gen}}(G_r)}
\newcommand{\grepfeven}{\Irr_{\rm{gen}}(G_{2n})}
\newcommand{\grepfodd}{\Irr_{\rm{gen}}(G_{2n-1})}

\newcommand{\temprepfr}{\Irr_{\rm{temp}}(G_r)}

\newcommand{\dsrepfr}{\Irr_{\rm{ds}}(G_r)}
\newcommand{\dsrepfh}{\Irr_{\rm{ds}}}
\newcommand{\dsrepfri}{\Irr_{\rm{ds}}(G_{r_i})}

\newcommand{\ntemprepfr}{\Irr_{\rm{ntemp}}(G_r)}

\begin{document}


\title {On the exterior square  $\varepsilon$-factors of $GL_n$} 
\author{Ravi~Raghunathan} 
\address{Department of Mathematics \\ 
         Indian Institute of Technology Bombay\\
         Powai, Mumbai 400076\\ India}
\email{raviraghuanthan@iitb.ac.in}


\begin{abstract} Let $\pi$ be a generic representation of $\glr(F)$, where $F$ is a $p$-adic field. We show using 
global methods that the $\varepsilon$-factors associated to the exterior square of $\pi$ via the 
Jacquet-Shalika integrals and Langlands-Shahidi methods coincide.  Along the way
we also give a new proof of the local functional equation in the $p$-adic case following the techniques in 
\cite{BP2021}.
\end{abstract}

\maketitle


\markboth{Ravi Raghunathan}{On the archimedean exterior square $L$-function}


\section{Introduction}\label{intro}

We state our main result first, referring the reader to Section \ref{secprelim} for the precise definitions 
and the notation used below.

Let $F$ be any local field of characteristic $0$, and let $G_r=\glr(F)$. Fix a non-trivial additive character $\psi: F\to \C$. 
Let  $\grepfr$ denote the isomorphism classes of generic representations of $G_r$. There are three 
ways to attach exterior square $L$-factors and $\varepsilon$- factors to a representation $\pi$ in $\grepfr$.
The {\it Galois} (or sometimes {\it arithmetic}) factors defined via
the local Langlands correspondence are denoted by $\lgepif$ and $\epgepif$ respectively, those defined
via the Langlands-Shahdi method by $\lshepif$ and $\epshepif$, and those defined
using the integrals of Jacquet and Shalika and a local functional equation by $\ljsepif$ and $\epjsepif$.

The main aim of this paper is to establish the following theorem.
\begin{theorem}\label{epjsequalepshthm} Let $F$ be a $p$-adic field, and let $\pi\in\grepfr$. 
Then
\begin{equation}\label{epjsequalepsheqn}
\epjsepif=\epshepif.
\end{equation}
\end{theorem}
In the archimedean case, the theorem above was recently proved by Jiang, Liu, Sun and Tian in Theorem 2.1 of 
\cite{JLST2025}. We will need to invoke this result to prove our theorem in the $p$-adic case.

The definition of $\epjsepif$ above is dependent on the existence of a local functional equation 
for the Jacquet-Shalika integrals.  When $F$ is $p$-adic, the local functional equation was established by Matringe 
\cite{Matringe2014} and Cogdell-Matringe \cite{CoMa2015} (see \cite{KeRa2012} for discrete series, and more generally, 
globalisable representations). A second aim of this paper is to reprove the local functional equation 
in the $p$-adic case imitating the global arguments used in \cite{BP2021} to treat the Asai $L$-function. 
We refer the reader to Theorems \ref{xientirethm} and \ref{padiclocfneqnthm}  where the precise form of the local 
functional equation appears.

In \cite{CoShTs2017} the authors established that $\epshepif=\epgepif$, when $F$ is $p$-adic field.
The corresponding theorem for archimedean $F$ is a much older result of Shahidi \cite{Shahidi85}. 
These results, together with those of \cite{JLST2025} and Theorem \ref{epjsequalepshthm} show 
that the three $\varepsilon$-factors must coincide.
\begin{theorem}\label{epjsequalepgthm} Let $F$ be a local field of characteristic $0$, and let $\pi\in\grepfr$. 
Then
\begin{equation}\label{epjsequalepg}
\epjsepif=\epshepif=\epgepif.
\end{equation}
\end{theorem}

The requirement that the $L$-functions and $\varepsilon$-factors defined in different ways coincide is 
intrinsic to the development of the theory of automorphic $L$-functions.
Indeed, the formulation of the Local Langlands Conjectures (LLC) for $G_r$ incorporates these equalities for the 
Rankin-Selberg $L$-functions as part of the statements. 
We wish to emphasise that proving the equalities for the exterior square factors is more subtle than
proving the analogues for the Rankin-Selberg or Asai factors, mainly because the archimedean theory of the 
integral representations of Jacquet and Shalika is more complicated and was previously less 
well developed in the former case (the Asai case can often be reduced to the Rankin-Selberg case
by making suitable assumptions about splitting at the archimedean places).
The second tool that was missing (and is now a consequence of Theorem \ref{epjsequalepshthm}) is 
the ability to express the factors $\epjsepif$ of generic 
representations in terms of the $\varepsilon$-factors of the inducing data.
Such formul{\ae} were available not only in the Rankin-Selberg and Asai situations, but also for
the factors $\epshepif$. Thus, the theory of the factors $\epjsepif$ really is somewhat more complicated.

\subsection{Organisation of the paper and an outline of the main ideas} 
A word about our methods of proof as well the logical flow of arguments in this paper. Section
\ref{secprelim} is largely concerned with notation and with recalling the basic notions of the representation
theory of reductive groups over a local field including the notions of the Harish-Chandra Schwartz space and
the description of the Fell topology on the space of tempered representations. Tempered representations within 
a connected component in the latter space vary with a parameter $\lambda\in \C^k$, and Proposition \ref{whittanalytic}
(due to Beuzart-Plessis in \cite{BP2021}) establishes the existence of good analytic sections of the Harish-Chandra
Schwartz spaces in this parameter. Section \ref{seclshmethod} collects a number of known results for the Galois and 
Langlands-Shahidi exterior square factors. The deduction of the odd case of Theorem \ref{epjsequalepsheqn} from the 
even case, which is almost immediate, is explained here.
The precise definitions of the local and global Jacquet-Shalika integrals and some of their known properties are 
recalled in Section \ref{secjsint}.

In Section \ref{secanalyticity}, the analyticity of various $L$-factors, 
$\varepsilon$-factors, and the continuity of the Jacquet-Shalika integrals with respect to the Whittaker functions and
Schwartz functions use to define them are recorded. These results are taken from Section 3 of \cite{JLST2025} and 
are modelled on the treatment of the Asai integrals in Section 3.3 of \cite{BP2021}. 

Theorem \ref{henniart84} is an old globalisation result of Henniart in \cite{Henniart1984} (see also \cite{PrSc2008} for a more 
recent general relative version from which Theorem \ref{henniart84} can be deduced) which allows us to conclude 
that $\epjsepif=\epshepif$ for supercuspidal representations. This is carried out in Section \ref{secscrep}. 

The main new ingredient in our paper is the use of a consequence of Theorem 4.8 of \cite{Shin2012} which appears 
as Theorem \ref{shinglobthmtwo} in Section \ref{globmethpadic}. It can be used to assert that every neighbourhood 
(in the Fell topology) of a given tempered representation $\pi$ of $G_r$ 
contains a representation that can be globalised to a cuspidal automorphic representation of a number field $k$ which is either 
supercuspidal or unramified at all finite places, other than at the given place $v$ such that $k_v\cong F$, if the given 
place is finite. 
Since, the local $\varepsilon$-factors at the unramified can be taken (after a suitable choice of global additive 
character) to be identically $1$ and the equality of the $\varepsilon$-factors holds at the supercuspidal places and 
archimedean places, a comparison of the global functional equations of the global $L$-functions $\ljsePi$ and $\lshePi$ 
will allow us to conclude in Proposition \ref{denseeppadic} that Theorem \ref{padiclocfneqnthm} and Theorem 
\ref{epjsequalepshthm} are valid for a dense subset of the tempered representations of $G_{r}$.

The passage to all generic $p$-adic representations is done in two stages. We first extend the local functional 
equation to the class of nearly tempered representations using Corollary 2.7.1 of \cite{BP2021} (see Proposition 
\ref{whittanalytic}) which establishes the existence of good (analytic) sections of Whittaker functions, and which, in turn, 
allows us to use continuity and analyticity arguments on the connected components of the Bernstein variety. Also 
crucial is the technical Proposition 2.8.1 of \cite{BP2021} which allows one to extend analytic functions on products of 
right half-planes (in $\C$) and open sets in hyperplanes $P\subset\C^r$ to all of $\C\times P$. This allows us to to 
prove Theorem \ref{xientirethm} and to extend our proof of the local functional equation to all generic representations, 
proving Theorem \ref{padiclocfneqnthm}, and also Theorem \ref{epjsequalepshthm} in general.

\section*{Acknowledgements}  I would also like to thank U. K. Anandavardhanan for helpful discussions and for pointing 
out an error in an earlier version of this paper. It is also a pleasure to thank Dipendra Prasad for many helpful discussions.
I am grateful to Yeongseong Jo for bringing the paper \cite{JLST2025} to my attention.

\section{Preliminaries}\label{secprelim} 

\subsection{Notation}\label{subsecnot} We will lean heavily on \cite{KeRa2012} for our notation below.
We will denote by $F$ a fixed local field of characteristic $0$ and by $k$ a number field.
When $F$ is a $p$-adic field, its ring of integers of $F$ will be denoted $\of$ and the cardinality of its residue field by $q$ (or $q_F$ if it is necessary to specify the field).
Let $k_v$ be the completion of $k$ at a place $v$. 
When $v$ is finite, we let $\ov$ denote the ring of integers of $k_v$ and let $q_v$ be the cardinality 
of the residue field.
We let ${\rm G}_r$ stand for the group $\glr$. The $k_v$ points of ${\rm G}_r$ 
will be denoted $G_{r,v}$, and the $\ak$ points of ${\rm G}_r$ will be denoted 
by ${\rm G}_r(\ak)$, where $\ak$ is the ad\`ele ring of $k$. We will often
need to consider quotients of the form $\gabygk$, which
we denote by $[G]_k$. For the local field $F$,
we will denote $\glr(F)$ simply by $G_r$, as before. 
We follow these two conventions, whenever convenient, for any algebraic groups that may occur in this paper.

We let $B_r$ be the (standard Borel) subgroup of upper triangular matrices in $G_r$, 
$N_r$ the standard maximal unipotent subgroup of $G_r$, and $Z_r$ the center of $G_r$. Let 
$\matn$ be the space of all $n\times n$ matrices and $\liebn$ be the subspace of 
all upper triangular $n\times n$ matrices. We let $P_r$ be the mirabolic subgroup
of $G_r$, that is, it is the stabiliser in $G_r$ of the vector 
$(0,\ldots ,0,1)\in F^r$ under multiplication on the right. We see that 
\[
P_r=\left\{g\in {\rm G}_r(E)\Big\vert\,g=\begin{pmatrix} *&\cdots &*&*\\
							          \vdots& &\vdots&\vdots\\
							          *&\cdots &*&*\\
							          0&\cdots &0&1\end{pmatrix}\right\}.
							          \]
We also set
\[
w_n = \left( \begin{array}{ccc} &&1 \\ & \iddots & \\ 1&& \end{array} \right)
\quad\text{and}\quad w_{n,n} = \left( \begin{array}{cc} 0 &I_{n}\\ I_{n}& 0\end{array} \right).
\]
When $r=2n-1$, we define
\[
\tau=\begin{pmatrix} 0&w_n&0\\
                     w_n&0&0\\
                     0&0&1\end{pmatrix},
\]
viewed as an element in $G_r$, $G_{r,v}$ or ${\rm G}_r(\ak)$ as appropriate.

Let $F$ be a $p$-adic field. By a representation of $\pi$ of $G_r$ we will mean a smooth admissible complex representation of $G_r$. If $F$ is archimedean, 
we will mean a smooth admissible Fr\'echet representation of moderate growth in the sense of Casselman-Wallach, of finite length and with complex coefficients.
We will denote by $\tilde{\pi}$ denote the contragredient representation of $\pi$. 
We let $\irredrepr$, $\temprepfr$ and $\dsrepfr$ be the sets of isomorphism classes of all irreducible representations, irreducible tempered representations and irreducible square-integrable (or discrete series) representations of $G_r$ respectively. We let $\nu(g)=|\det g|$ for $g\in G_r$. If $\sigma\in \dsrepfr$, a representation of $G_r$ of the form $\sigma\nu^{s}$, $s\in \C$ is called a quasi-square integrable representation.

Let $\sfnf$ be the space of Schwartz functions on $F^n$. The Fourier transform of $\phi$ with respect to a suitably chosen additive character $\psi:F\to C$ is denoted $\widehat{\phi}$. The Haar measure on $F^n$ is chosen so that $\widehat{\widehat{\phi}}(x)=\phi(-x)$.
Likewise, let $ \cs(\ak^{n})$ be the space of 
Schwartz-Bruhat functions on $\ak^n$ with a Fourier transform defined with respect to a suitably chosen global additive character $\psi=\otimes_v^{\prime}\psi_v:\ak\to \C^{*}$.
See \cite{KeRa2012} for the relevant choices of the additive characters and Haar measures.

For $\nu\in \R\cup\{-\infty\}$ we define the 
(right) half-plane $\uh_{\nu}=\{z\in \uh\,\vert\, \re{z}>\nu\}$. If $[\nu_1,\nu_2]\subset \R$ is an interval, we define the vertical strip $\uh_{[\nu_1,\nu_2]}$ to be the closure of the set $\uh_{\nu_1}\cap \uh_{\nu_2}^c$. Let $M$ be a
complex manifold. Following Section 2.8 of \cite{BP2021}, we say that a holomorphic function $f:\uh_{\nu}\times M\to \C$ is of finite order on vertical 
strips in the first variable locally uniformly in the second variable if there exists $d\in \R$, such that for every vertical strip $V\subset \uh_{\nu}$ and every 
compact set $C\subset M$, 
\[
\sup_{(s,t)\in \uh_{\nu}\times C}e^{-|s|^d}\vert f(s,t)\vert <\infty.
\]

\subsection{Induced representations and the topology on $\temprepfr$}

Let $r_i\in \N$, $1\le i\le k$ be such that $r_1+r_2+\cdots +r_k=r$ and let
$\sigma_i\in \dsrepfri$, so $\sigma=\sigma_1\times\sigma_2\times \cdots \sigma_k$
is in $\dsrepfh(M)$, where $M=G_{r_1}\times G_{r_2}\times \cdots \times G_{r_k}$.
We let $P$ be a standard parabolic in $G_r$ with $M$ as its Levi subgroup.
For any $\lambda=(\lambda_1,\cdots \lambda_k)$, we define
\begin{equation}\label{pilambdadef}
\pi_{\lambda}:=\Ind_P^{G_r}(\sigma_{\lambda}):=\Ind_{P}^{G_{r}}(\sigma_{1}\nu^{\lambda_1}\otimes\sigma_{2}\nu^{\lambda_2}\otimes \cdots \sigma_{k}\nu^{\lambda_k}),
\end{equation}
where $\Ind_P^{G_r}$ denotes the smooth normalised parabolic induction from $P$ to $G_r$. Alternatively,
we may write
\[
\pi_{\lambda}=\Ind_P^{G_r}(\Delta_1\times\cdots \Delta_k),
\]
where the $\Delta_i=\sigma_i\nu^{\lambda_i}$ are quasi-square-integrable representations, $1\le i\le k$.
Recall that if $\re{\lambda_i}=0$ for all $1\le i\le k$, $\pi_{\lambda}$ is irreducible and tempered, and 
every irreducible tempered 
representation necessarily has this form.
If $|\re{\lambda_i}|<1/4$ for all $1\le i\le k$, we will say (following \cite{BP2021}) that 
$\pi_{\lambda}$ is nearly tempered and denote the set of isomorphism
classes  of irreducible nearly tempered representations by $\ntemprepfr$. 

We follow the description of the topology on $\temprepfr$ on page 23 of \cite{BP2021}. Fix a standard parabolic subgroup $P$ with Levi decomposition $P=MU$ of $G_r$ as above. For $\lambda\in i\R^k$, the map $\lambda\mapsto \pi_{\lambda}$ identifies a quotient of $i\R^k$ with a subset of $\temprepfr$ (which is equipped with the quotient topology). The topology on $\temprepfr$ is the one for which these subsets are precisely its connected components as $P$ varies over all standard parabolic subgroups.

We now follow the notation in Section 2.1 of \cite{BP2021}.
For a connected reductive group ${\rm G}$ over $F$, let $A_{G}$ be the maximal split torus in the centre 
of $G$ and let $X^{*}(G)$ and $X^{*}(A_G)$ be the groups of algebraic characters of $G$ and $A_G$ respectively. We set
\begin{align}
{\mathcal A}_G^{*}&=X^{*}(A_G)\otimes\R, \quad{\mathcal A}_{G,\C}^{*}=X^{*}(A_G)\otimes\C,\\
{\mathcal A}_G&=\Hom(X^{*}(G),\R),\quad\text{and}\quad {\mathcal A}_{G,\C}=\Hom(X^{*}(G),\C).
\end{align}
If $P=MU$, as above we may restrict a character
$\chi\in X^{*}(A_{M})$ to $A_{G}$. This yields surjections from 
${\mathcal A}_M^{*}$ to ${\mathcal A}_{G}^{*}$ and from ${\mathcal A}_{M,\C}^{*}$
to ${\mathcal A}_{G_r,\C}^{*}$. The kernels of these two surjections will be
denoted ${({\mathcal A}_M^{G})}^{*}$ and ${({\mathcal A}_{M,\C}^{G})}^{*}$ respectively.
We let $\langle .\,,.\rangle$ be the natural pairing between ${\mathcal A}_{G}^{*}$ and ${\mathcal A}_{G}$
(or between ${\mathcal A}_{G,\C}^{*}$ and ${\mathcal A}_{G,\C}$)
and define $H_G: G\to {\mathcal A}_{G}$ by $\langle \chi(g)\,,H_G(g)\rangle=\log\vert\chi(g)\vert_F$.
For $\lambda \in {\mathcal A}_{G_r,\C}^{*}$, define the unramified character $g^{\lambda}$ on $G$ by
$g^{\lambda}:=e^{\langle \lambda,H_G(g)\rangle}$.

We let $\Delta$ be the set of simple roots of $T_r$ in $N_r$.
We equip ${\mathcal A}_{T_r}^{*}$ with the following partial order: $\lambda\preceq\mu$ if
and only if $\mu-\lambda$ is a linear combination of simple roots with non-negative coefficients. 
Let $W^{G_r}$ be the Weyl group of $G_r$; more specifically, $W^{G_r}={\rm{Norm}}_{G_r}(T_r)/T_r$.
Let $\overline{({\mathcal A}_{T_r}^{*})^{+}}$ be the closure of the negative Weyl chamber in 
${\mathcal A}_{T_r}^{*}$, that is
\[
\overline{({\mathcal A}_{T_r}^{*})^{+}}=
\{\lambda\in {\mathcal A}_{T_r}^{*}\,\vert\, \text{$\langle \lambda,\alpha^{\vee}\rangle\le 0$ for all $\alpha^{\vee}\in \Delta$}\}.
\]
For $\lambda\in {\mathcal A}_{T_r}^{*}$, we denote by $|\lambda|$ the unique element in 
$W^{G_r}\lambda\cap \overline{({\mathcal A}_{T_r}^{*})^{+}}$.
For $\mu\in ({\mathcal A}_{T_r}^{G_r})^{*}$, it will be useful to define the open subset 
\[
{\mathcal U}[\prec \mu]:=\{\lambda\in ({\mathcal A}_{M,\C}^{G_r})^{*}\,\big\vert\,  |\re{\lambda}|\prec \mu\}.
\]
of $({\mathcal A}_{M,\C}^{G_r})^{*}$. 

We summarise very briefly the discussion on holomorphy on page 9 of \cite{BP2021}.
If $K$ is a standard maximal compact subgroup of $G$, there is a (topological) isomorphism between the spaces
$t_{\lambda}:\pi_{\lambda}\to \pi_K=\Ind_{K\cap P}^K\sigma\vert_M$ given by restricting $\pi_{\lambda}$ to $K$.
A map $\lambda (\in{\mathcal A}_{M,\C}^{*})\to e_{\lambda}\in\pi_{\lambda}$ is called an analytic or holomorphic section if
$t_{\lambda}\circ e_{\lambda}:{\mathcal A}_{M,\C}^{*}\to \Ind_{K\cap P}^K\sigma\vert_M$ is an analytic function.
The notion of holomorphy is independent of the choice of $K$.
 If $F$ is archimedean, we denote by $\pi_{\lambda}^{\prime}$ the topological dual of $\pi_{\lambda}$, while for 
$F$ $p$-adic, we take $\pi_{\lambda}^{\prime}$ to be the contragredient $\tilde{\pi_{\lambda}}$.
A map $\lambda\to T_{\lambda}\in \pi_{\lambda}^{\prime}$ is said to be analytic if for every analytic
section $e_{\lambda}$, the map $\lambda\to T_{\lambda}(e_{\lambda})$ is analytic.

We will only be dealing with ${\rm G}={\rm G}_r$ and the Levi subgroups of its (standard) parabolic subgroups.
For the standard Borel subgroup we have the Levi decomposition $B_r=T_rN_r$, where $T_r$ is the maximal torus in $G_r$. 
The corresponding real vector space ${\mathcal A}_{T_r}^{*}=X^{*}(T_r)\otimes \R$ can be identified with 
$\R^r$, and ${\mathcal A}_{T_r,\C}^{*}$ can be identified with $\C^r$.

When $P=MU$, as above, we see that ${\mathcal A}_M^{*}\cong \R^k$ and ${\mathcal A}_{M,\C}^{*}\cong
\C^k$. More specifically, we have the identification 
\[
{\mathcal A}_{M}^{*}=\{\lambda=(\underbrace{\lambda_1,\ldots \lambda_1}_{\text{$r_1$ times}},\ldots ,
\underbrace{\lambda_k,\ldots \lambda_k}_{\text{$r_k$ times}})\,\vert\, \lambda_i\in \R, 1\le i\le k\}\cong\R^k
\hookrightarrow {\mathcal A}_{T_r}^{*}\cong \R^r,
\]
and similarly the space ${\mathcal A}_{M,\C}^{*}$ can be 
identified with a $k$-dimensional subspace of ${\mathcal A}_{T_r,\C}^{*}\cong\C^r$. This inclusion arises because
$A_M\subset T_r$, so a character on $T_r$ restricted to $A_M$ yields a character of $A_M$. 
Note that ${({\mathcal A}_M^{G_r})}^{*}$ can be identified with the $(k-1)$-dimensional subspace of
of $\R^k$ given by $\lambda_1+\lambda_2+\cdots +\lambda_k=0$, while ${({\mathcal A}_{M,\C}^{G_r})}^{*}$ can
be identified with the corresponding $(k-1)$-dimensional subspace of $\C^k$.

\subsection{Generic representations}
We retain the notation of the previous two subsections. We will largely follow \cite{BP2021} for the exposition below.
Let $ \psi:F\to \C$ be a nontrivial additive character of $F$ and view it as a character of $N_{r}$ by setting
$\psi(u) := \displaystyle{\prod_{j=1}^{r-1}}\psi(u_{j,j+1})$ for $u\in N_{r}$. When $F\cong k_v$, we will denote
the character $\psi$ by $\psi_v$ and obtain a corresponding character of $N_{r,v}$. Such a character $\psi$ is called a generic character of $N_r$. 

Let $\pi\in \irredrepr$ and $\psi$ be a generic character of $N_r$.
A Whittaker functional on $\pi$ is a continous linear functional
$J:\pi\to \C$ such that $J(\pi(ug)v)=\psi(u)J(\pi(g)v)$ for all $v\in \pi$, 
$u\in N_r$ and $g\in G_r$. A representation $\pi\in \irredrepr$ is said to be generic if there exists a Whittaker 
functional on $\pi$. The isomorphism classes of generic representations of $G_r$ will be denoted $\grepfr$.
If $\pi \in \grepfr$, it is known by Theorem 6.2f of \cite{Vogan1978} when $F$ is archimedean, and by 
Theorem 9.7 of \cite{Zelevinsky1980} when $F$ is $p$-adic, that it is necessarily of the form $\pi_{\lambda}$ described above.

We define ${\mathcal{C}}(N_r\backslash G_r,\psi)$ to tbe the space of all smooth
functions $W:G_r\to \C$ such that $W(ng)=\psi(n)W(g)$ for all $n\in N_r$ and 
$g\in G_r$.
For $\mu\in \C^k$, we denote the Harish-Chandra Schwartz space of 
functions on $G_r$ by $\hcsfmu$. It is the subspace of ${\mathcal{C}}(N_r\backslash G_r,\psi)$, the 
space of functions $W:G_r\to \C$ satisfying a certain moderate 
growth condition depending on the Harish-Chandra $\Xi$-function $\Xi_\mu^{G_r}$. When $F$ is 
$p$-adic, the functions $W$ are assumed to be 
right-invariant by a compact open subgroup, while when $F$ is archimedean the functions are assumed smooth. 
When $F$ is $p$-adic, $\hcsfmu$ is a strict LF space, while when $F$ is archimedean, $\hcsfmu$ is a 
Fr\'echet space (see Section 2.5 of \cite{BP2021} for a proof of this fact as well as for the precise 
definition of the moderate growth condition and other details). 

A Whittaker functional $J$ induces
a continuous $G_r$-equivariant map
$\tilde{J}:\pi\to {\mathcal C}(N_r\backslash G_r,\psi)$. When $\pi=\pi_{\lambda}$, it is known (by Section 15.4 
of \cite{Wallach1992}, \cite{CaSh1980}) that the construction of
Jacquet's integral yields an analytic family of Whittaker functionals 
$\lambda\to J_{\lambda}\in \Hom_{N_r}(\pi_{\lambda},\psi)$ which is everywhere non-vanishing. The Whittaker 
functional $J_{\lambda}$ induces a continuous $G_r$-equivariant map
$\tilde{J_{\lambda}}:\pi_{\lambda}\to {\mathcal C}(N_r\backslash G_r,\psi)$.
The space of functions $g\mapsto J_{\lambda}(\pi_{\lambda}(g)v)$, $v\in \pi$, is called the Whittaker model of 
$\pi_{\lambda}$ and is denoted $\wpiflambda$. 

For each $W\in \wpif$, we define the function $\widetilde{W}$ on $G_r$ by 
$\widetilde{W}(g) = W(w_{r}\,^{\iota}g)$, where $~^{\iota}g =~^tg^{-1}$. 
The Whittaker model of the contragredient representation 
$\tilde{\pi}$, $\wpidualf$, consists precisely of 
the set of functions $\{\widetilde{W}\,|\,W_v\in \wpif\}$. We will denote by 
$\rho(g)\widetilde{W}$ the right translation of the function 
$\widetilde{W}$ by $g$. 

The following proposition (Corollary 2.7.1 of \cite{BP2021}) will be essential 
in what follows.
\begin{proposition} \label{whittanalytic} For every $\lambda_0\in ({\mathcal A}_{M,\C}^{G_r})^{*}$ and 
$W_0\in \wpiflambdao$, there exists a map
\[
\lambda( \in {\mathcal A}_{M,\C}^{G_r})^{*}\mapsto W_{\lambda}\in \wpiflambda
\]
such that
\begin{enumerate}
 \item for every $\mu\in ({\mathcal A}_{T_r}^{G_r})^{*}$ and 
 $\lambda\in {\mathcal U}[\prec \mu]:=\{\lambda\in ({\mathcal A}_{M,\C}^{G_r})^{*}\big\vert\, \vert\re{\lambda}\vert \prec \mu\}$
 we have $W_{\lambda}\in \hcsfmu$ and the map 
 \[
  \lambda\in {\mathcal U}[\prec\mu]\mapsto W_{\lambda}\in \hcsfmu
 \]
obtained by composing with this inclusion is analytic, and
\item $W_{\lambda_0}=W_0$.
\end{enumerate}
\end{proposition}

\subsection{Notation for the partial global integrals, $L$- and $\varepsilon$-factors}
Let $k$ be a number field and $S$ be a set of places of $k$. Of special interest will be the 
set $S=S_{\infty}$ of all the infinite places of $k$. If we are given functions $F_v(s)$ or $F(s,v)$, $s\in \C$ for each
place $v$ in $S$, we define
\[
F_S(s)=\prod_{v\in S}F_v(S)\quad\text{and}\quad F^S(s)=\prod_{v\notin S} F_v(s)
\]
when they converge.

For every generic representation $\pi_v$ of $\glr(k_v)$ we have three definitions 
of the twisted $L$- and $\varepsilon$-factors: the Galois factors $\lgepiv$ and $\epgepiv$, the
 Langlands-Shahidi factors
$\lshepiv$ and $\epshepiv$ defined via the Langlands-Shahidi method, and the factors $\ljsepiv$ and $\epjsepiv$ 
defined via the integral representations of Jacquet and Shalika. If $\Pi=\otimes_v^{\prime}\Pi_v$ is a globally generic 
admissible representation and $S$ is a finite set of places, we define the partial $L$-functions and $\varepsilon$-factors
\[
L_{*,S}(s,\Pi,\wedge^2)=\prod_{v\in S}L_{*}(s,\Pi_v,\wedge^2), L_{*}^S(s,\Pi,\wedge^2)=\prod_{v\notin S}L_{*}(s,\Pi_v,\wedge^2),
\]
and 
\[
\varepsilon_{*,S}(s,\Pi,\wedge^2,\psi)=\prod_{v\in S}\varepsilon_{*}(s,\Pi_v,\wedge^2,\psi_v),
\]
where $L_{*,S}(s,\pi,\wedge^2)$, $ L_{*}^S(s,\Pi,\wedge^2)$ and $\varepsilon_{*,S}(s,\Pi,\wedge^2,\psi)$ can denote 
any of the three partial $L$- or $\varepsilon$-factors according to whether $*=G$, $Sh$ or $JS$. 
It will also be useful to define the $\gamma$-factors 
\begin{equation}\label{galgammadefn}
\gamma_{*}(s,\pi,\wedge^2,\psi)=\frac{L_{*}(1-s,\tilde{\pi},\wedge^2)}{L_{*}(s,\pi,\wedge^2)}\varepsilon_{*}(s,\pi,\wedge^2,\psi),
\end{equation}
and
\begin{equation}\label{shgammadefn}
\gamma_{*,S}(s,\Pi,\wedge^2,\psi)=\frac{L_{*,S}(1-s,\widetilde{\Pi},\wedge^2)}{L_{*,S}(s,\Pi,\wedge^2)}
\varepsilon_{*,S}(s,\Pi,\wedge^2,\psi)
\end{equation}
Because Theorem \ref{henniart} guarantees the equality of the local $L$-functions, we see that establishing
the equality of the $\varepsilon$-factors is equivalent to establishing the equality of the $\gamma$-factors.
                                      
\section{The Galois and Langlands-Shahidi $L$- and $\varepsilon$-factors}\label{seclshmethod}
We quickly review the relevant theorems for the Galois and Langlands-Shahidi exterior square factors.

\subsection{The local theory} 
We recall the definition of the Galois exterior square factors.
As before, let $\psi:F\to \C$ be a fixed additive character and let $\pi\in \grepfr$. The local Langlands
correspondence associates an $r$-dimensional representation $\sigma$ of the Weil-Deligne group 
$W_F^{\prime}$. We define 
\[
\lgepif=L(s,\wedge^2(\sigma))\quad\text{and}\quad \epgepif=\varepsilon(s,\wedge^2(\sigma),\psi),
\]
where the $L$- and $\varepsilon$-factors on the right hand sides of the equalities above are the Artin factors. 

We will not recall the definitions of the Langlands-Shahidi factors, but record the following results below.
Combining the results of \cite{Shahidi85} (for $F$ archimedean) and \cite{Henniart2010} (for $F$ $p$-adic), we 
have the following theorem.
\begin{theorem}\label{henniart}
For $\pi\in \grepfr$,
\[ 
\lshepif = \lgepif.
\]
\end{theorem}
\begin{remark} The $L$-functions $\lshepif$ and $\lgepif$ are defined for all irreducible admissible representations
of $G_r$ and the Theorem \ref{henniart} holds for this larger class.
\end{remark}
When $F$ is a $p$-adic field, the factors $\epshepif$ have the form 
$w_{Sh}(\pi)q^{-\kappa_{Sh}(\pi)s}$, where
$w_{Sh}(\pi)\in \C$ and $\kappa_{Sh}(\pi)\in \N$. 
When $F$ is an archimedean field, the $\varepsilon$-factors $\epshepif$ have the form 
$w_{Sh}(\pi)B^{-\kappa_{Sh}(\pi)s}$, where $w_{Sh}(\pi)\in \C$, $B\in \R_{>0}$, and 
$\kappa_{Sh}(\pi)s\in \Z_{\ge 0}$. For a suitably normalised choice of the character $\psi$, we can 
take $B=1$. The following result was proved in \cite{Shahidi85}.
\begin{theorem}\label{shahidi85} Let $F$ be an archimedean local field, and let
$\pi$ be a smooth irreducible representation of $G_r$. Then
\[ 
\epshepif = \epgepif.
\]
\end{theorem}

We now describe the factors of (generic) parabolically induced representations in terms of the
factors of the inducing data. For the Galois factors this a straightforward consequence of the formalism of the 
Artin $L$-functions, while for the Langlands-Shahidi factors, this is a consequence of the multiplicativity of the 
gamma factors proved in \cite{Shahidi1990a} and \cite{Henniart2010} in conjunction with the fact that the LLC preserves pairs of the local factors (\cite{HaTa2001}, \cite{Henniart2000}).
\begin{proposition}\label{inductextsq} Let $\pi_{\lambda}=\Ind_P^{G_r}\sigma_{\lambda}=\Ind_P^{G_r}(\Delta_1\times\cdots \Delta_k)$,
where the $\Delta_i=\sigma_i\nu^{\lambda_i}$ are quasi-square-integrable representations, $1\le i\le k$.
For $*=G$ or $*=Sh$, and every $\lambda\in\C^k$, we have 
\begin{equation}\label{lfnofparind}
L_{*}(s,\pi_{\lambda},\wedge^2)=\prod\limits_{i=1}^{k}
L_{*}(s+2\lambda_i,\sigma_{i},\wedge^2)\prod\limits_{1\leq i<j\leq k}
L_{*}(s+\lambda_i+\lambda_j,\sigma_{i}\times\sigma_{j})
\end{equation} 
and
\begin{equation}\label{epfnofparind}
\varepsilon_{*}(s,\pi_{\lambda},\wedge^2,\psi)=\prod\limits_{i=1}^{k}
\varepsilon_{*}(s+2\lambda_i,\sigma_{i},\wedge^2,\psi)\prod\limits_{1\leq i<j\leq k}
\varepsilon_{*}(s+\lambda_i+\lambda_j,\sigma_{i}\times\sigma_{j},\psi).
\end{equation}
\end{proposition}

\begin{remark} Proposition \ref{inductextsq} was proved separately for the Galois and Langlands-Shahidi exterior square
 factors, and was an ingredient in the proofs showing the equality of the two sets of factors. 
\end{remark}

\subsection{The Global Theory}\label{lsmglobal}

Let $\Pi = \otimes'_{v}\Pi_{v}$ be a unitary cuspidal representation of $G_r(\ak)$ and let 
$\psi=\otimes'_v\psi_v$ be a global additive character trivial on $k$. We define the completed (global) Langlands-Shahidi $L$-function as
\begin{equation*}
\lshePi:= \prod_{v}\lshePiv
\end{equation*} 
and the global Langlands-Shahidi $\varepsilon$-factor as
\[
\epshePi:=\prod_{v}\epshePiv,
\]
where $v$ runs over all the absolute values of $k$. It is a fact that the global $\varepsilon$-factor does not depend
on the choice of $\psi$ even though the local factors depend on the choices of the local additive characters $\psi_v$. It is obviously entire and non-vanishing on $\C$.
Combining Propositions 3.1 and 3.4 of \cite{Kim99} we obtain:
\begin{proposition}
The function $\lshePi$ is holomorphic for $\text{Re}(s) > 1$.
\end{proposition}
Theorem 7.7 of \cite{Shahidi90} gives further information on the global $L$-function.
\begin{theorem}\label{lshfnaleqnthm}
The $L$-function $\lshePi$ admits a meromorphic continuation to the entire complex plane 
and satisfies a functional equation
\begin{equation}\label{lshfnaleqn}
\lshePi = \epshePi L_{Sh}(1-s, \tilde{\Pi}, \wedge^{2}),
\end{equation} 
where $\tilde{\Pi}$ denotes the representation contragredient to $\Pi$.
\end{theorem}
More precise information of about the holomorphy of $\lshePi$ is available (indeed, it is almost always entire -- 
see \cite{Kim99} and Corollary 7.5 of \cite{KeRa2012} for example) but will not be required in this paper.

\section{The Jacquet-Shalika integrals: known results}\label{secjsint}
We collect some of the results involving the Jacquet-Shalika integrals that are already known and record them below.

\subsection{The local integrals}\label{localintegralssubsec}

Let $\pi$ be an generic representation of $G_r$ with Whittaker model $\wpif$.
Jacquet and Shalika (in \cite{JaSh90a}) defined two families of integrals $\jswphif$ and $\jswf$ depending on whether $r$ is even or odd.

If $r$ is even, say $r = 2n $, we let
\begin{flalign}\label{J(even)}
\jswphif = 
\int_{N_{n} \backslash G_{n}}\int_{\lieb_{n}\backslash \matn} &W\left(\sigma\begin{pmatrix} I_{n} & X \\ 0 & I_{n} \end{pmatrix}
\begin{pmatrix} g & 0 \\ 0 & g \end{pmatrix}\right)\nonumber\\
&\psi(-\tr X)dX \phi(e_{n}g) |{\det g}|^{s} dg,
\end{flalign}
for each $W\in\wpif$ and $\phi$ in $\sfnf$, where $s$ is in $\C$, and 
$\sigma $ is the permutation given by
\[
\sigma = \begin{pmatrix}
1 & 2 & \cdots & n & | & n+1 & n+2 & \cdots & 2n \\
1 & 3 & \cdots & 2n-1 & | & 2 & 4 & \cdots & 2n 
\end{pmatrix}.
\]

If $r$ is odd, say $ r = 2n-1$, we define
\[
W_{1}(g)=\int_{F^n}W\left(g\begin{pmatrix}I_n&0&0\\
                                                              0&I_n&0\\
                                                              0&Z&1\end{pmatrix}\right)
\phi(Z)dZ,\,\,\text{and}
\]
\[
W_{2}(g)=\int_{F^n}W\left(g\begin{pmatrix}I_n&0&Y\\
                                                              0&I_n&0\\
                                                              0&0&1\end{pmatrix}\right)
\widehat{\phi}(Y)dY,
\]
for $W\in \wpif$ and $\phi\in \sfnf$.
Note that both the integrals above are (finite) sums of right-translates of $W$, so $W_1,W_2\in\wpif$.
We now define the integrals
\begin{flalign}\label{J(odd)}
\jswf = 
\int_{N_{n} \backslash G_{n}}\int_{\liebn\backslash \matn}&W\left(\sigma\begin{pmatrix} I_{n} & X & 0\\ 0 & I_{n} & 0\\ 0 & 0 & 1 \end{pmatrix}\right)
\left(\begin{pmatrix}g & 0 & 0 \\ 0 & g & 0\\ 0 & 0 & 1 \end{pmatrix}\right)\nonumber \\
 &\psi(-\tr X)dX |{\det g}|^{s-1} dg, 
\end{flalign}
for each $W$ in $\wpif$, where  $ \sigma $ is the permutation given by
\[
\sigma = \begin{pmatrix}
1 & 2 & \cdots & n & | & n+1 & n+2 & \cdots & 2n & 2n-1 \\
1 & 3 & \cdots & 2n-1 & | & 2 & 4 & \cdots & 2n & 2n-1
\end{pmatrix}.
\]
Proposition 1 of Section 7 and Proposition 3 of Section 9 of \cite{JaSh90a} together show
\begin{proposition} \label{jsintconv}
Let $\pi\in \operatorname{Irr}_{\rm{gen}}(G_{r})$ be a unitary generic representation of $G_n$. There exists $\eta > 0 $ 
such that the integrals $\jswphif$ (resp. $\jswf$) converge absolutely for 
$\text{Re}(s) >  1 - \eta$ for every $W\in \wpif$ and $\phi\in \sfnf$ .
\end{proposition}
When $F$ is a $p$-adic field the integrals yield rational functions of 
$q^{-s}$. When $F$ is archimedean, the following observation was made
in \cite{KeRa2012} (see the proof of Proposition 3.9 of that paper): the integrals $\jswvphiv$ (resp. $\jswv$) 
are finite sums of products of entire functions and Tate integrals. Consequently, they have 
meromorphic continuations to the whole plane. Moreover, as $W$ varies in $\jswf$ and $\phi$ varies in $\sfnf$  (resp. $W$ varies 
in $\jswf$), the set of exponents of the quasi-characters that appear in the Tate integrals in the expression for in 
$\jswphif$ (resp. $\jswf$) are finite in number by Proposition 6 of \cite{JaSh90a}). These exponents determine the possible 
poles of the meromorphic functions defined by these integrals. Thus, the set of poles is a finite union of translates of arithmetic progressions lying on lines parallel to the $x$-axis. 
 
 For the most part, we will need only the following weaker statement that follows from Proposition 6 of \cite{JaSh90a}
 (see also \cite{KeRa2012} and \cite{Belt2011} for the archimedean case).
 \begin{proposition}\label{archmeromorph} Let $F$ be a local field of characteristic $0$ and
 $\pi\in \operatorname{Irr}_{\rm{gen}}(G_{r})$. The integrals $\jswphif$ (resp. $\jswf$) have meromorphic continuations to 
 $\C$ for every $W\in \wpif$ and $\phi\in \sfnf$ (resp. every $W\in \wpif$).
\end{proposition}

When $F$ is $p$-adic, let ${\mathcal I}$ be the fractional ideal of 
$\C[q^s,q^{-s}]$ in $\C(q^{-s})$ generated by the integrals $\jswphif$ (resp. $\jswf$) as $W$ varies over all vectors in 
$\wpif$ and $\phi$ varies over all functions in  $\sfnf$ (resp. $W$ varies over all vectors in $\wpif$). Since 
$\C[q^s,q^{-s}]$ is a principal ideal domain, ${\mathcal I}$ must be generated by a single element. We will now need to use the following proposition of Belt \cite{Belt2011} (see also \cite{MiYa2013} for another proof in the 
$p$-adic case, and Proposition 3.7 of \cite{JLST2025} for a weaker statement in the archimedean case 
with a simpler proof).
\begin{proposition}\label{archnonvanish} Let $F$ be a local field and let 
$\pi\in\operatorname{Irr}_{\rm{gen}}(G_{r})$. For each $s\in \C$, there exist 
$W\in \wpif$ and $\phi\in \sfnf$ (resp. every $W\in \wpif$) such that $\jswphif$ is holomorphic at $s$ and 
$\jswphif\ne 0$ (resp. $\jswf$ is holomorphic and $\jswf\ne 0$).
\end{proposition}
In the $p$-adic case, Proposition \ref{archnonvanish} guarantees that the generator of ${\mathcal I}$ 
can be taken to have 
the form $1/P(q^{-s})$ for a polynomial $P(X)\in \C[X]$ with constant coefficient $1$. We set 
$\ljsepi=1/P(q^{-s})$. In \cite{KeRa2012} we proved the equality of $L$-factors stated below 
assuming the validity of the fomula \eqref{lfnofparind} for $\ljsepif$.
Subsequently, Jo gave a purely local proof of the equality of $L$-factors in \cite{Jo2020} while also proving the formula 
\eqref{lfnofparind} for $\ljsepif$along the way.
\begin{theorem}\label{lepiequalthm} Let $F$ be a $p$-adic field. If $\pi\in \operatorname{Irr}_{\rm{gen}}(G_{r})$,
\[
 \ljsepif=\lgepif.
\]
\end{theorem}
Combining the above result with Theorem \ref{henniart} allows us to conclude that when $F$ is $p$-adic, all the three $L$-functions coincide:
\begin{equation}\label{allthreelfnequal}
 \ljsepif=\lgepif=\lshepif
\end{equation}
\begin{remark}\label{allthreelfnequalrem} In light of the theorem above, we will simply use the single notation $\lepif$ for any of the three exterior square $L$-functions from this point onwards.
\end{remark}

Let $\wpif$ be the Whittaker model of $\pi$, $\wpidualf$ be the Whittaker model of its contragredient $\tilde{\pi}$,
and $\sfnf$. For $W\in \wpif$, we let $\tilde{W}$ be the element of $\wpidualf$ defined as $W(w_r{^tg^{-1}})$. 
Let $\rho$ denote the 
right translation action of $G_r$, $w_{n,n}$ a certain element of the Weyl group of 
$G_r$ when $r=2n$,  and $W_1$ and $W_2$ certain finite sums of right translates of $W\in \wpif$ when $r=2n-1$. 
It will be useful to set
\begin{flalign}
\Xi(s,W,\phi)&=\frac{\jswphif}{\lgepif}\quad\left(\text{resp.}\quad\Xi(s,W)=\frac{\jswf}{\lepif}\right)\\
\widetilde{\Xi}(s,W, \phi)&=\frac{J(s, \rho(w_{n,n})\widetilde{W}, \widehat{\phi})}{L(s, \tilde{\pi}, \wedge^{2})}
\quad\left(\text{resp.}\quad \widetilde{\Xi}(s,W_1)=
\frac{J(s, \widetilde{W_1})}{L(s, \tilde{\pi}, \wedge^{2})}\right)
\end{flalign}

\begin{theorem}\label{xientirethm} Let $F$ be a local field of characteristic $0$, and let $\pi\in \grepfeven$ 
(resp. $\grepfodd$) 
Then $\Xi(s,W,\phi)$ (resp. $\quad\Xi(s,W)$) is entire for all $W\in \wpif$ and $\phi\in \sfnf$ (resp. for all $W\in \wpif$). 
\end{theorem}
The theorem above shows that all the poles of the family of Jacquet-Shalika integrals must necessarily be poles of the 
$L$-function $\ljsepif$. This consequence of a theorem of \cite{Belt2011} appears in \cite{KeRa2012}.
In the archimedean case, it is part of Theorem 2.1 of \cite{JLST2025}.

Finally, we have local functional equations. For the square integrable representations, or more generally for 
globalisable generic representations, this was proved in \cite{KeRa2012}.
The methods of this paper will actually allow us to extend that proof to all generic representations in the $p$-adic case,
an approach we will pursue en route to the proof of our main theorem. The local functional equations
have been proved earlier by purely local methods by \cite{Matringe2014} and \cite{CoMa2015}.
The functional equation in the archimedean case was proved quite recently in \cite{JLST2025}.
These results are summarised below in Theorems \ref{padiclocfneqnthm} and 
Theorem \ref{archepequalthm}.
\begin{theorem}\label{padiclocfneqnthm} Let $F$ be a $p$-adic field and suppose that $\pi\in \grepfr$.
\begin{enumerate}
\item Let $r=2n$. There is a function $\epjsepif$ independent of $W$ and $\phi$ such that
\begin{equation}\label{padicevelocfneqn}
\widetilde{\Xi}(1-s,W, \phi)
= \epjsepif \Xi(s,W,\phi)
\end{equation}
for all $W\in \wpif$ and $\phi\in \sfnf$.
\item Let $r=2n-1$. There is a function $\epjsepif$, independent of $W$ and $\phi$, such that
\begin{equation}\label{padicoddlocfneqn}
\widetilde{\Xi}(1-s,W_1)
= \epjsepif \Xi(s,W_2)
\end{equation}
for all $W\in \wpif$.
\end{enumerate}
In both the cases above, $\epjsepif$ has the form $w_{JS}(\pi)q^{-\kappa_{JS}(\pi)s}$,
where $w_{JS}(\pi)\in \C$ and $\kappa_{JS}(\pi)\in \Z_{\ge 0}$.
\end{theorem}

We restate Theorem 2.1 of \cite{JLST2025} as
\begin{theorem}\label{archepequalthm} Let $F$ be an archimedean local field. Then
\begin{equation}
\epjsepif=\epgepif.
\end{equation}
\end{theorem}
In light of Theorem \ref{shahidi85}, when $F$ is archimedean, we obtain 
\begin{equation}\label{archallepequaleqn}
\epjsepif=\epgepif=\epshepif.
\end{equation}

In addition, we will use Proposition 2 of Section 7 
and Proposition 4 of Section 9 of \cite{JaSh90a}) for unramified representations which we summarise below.
\begin{proposition}\label{unramvector} Suppose that $F$ is a $p$-adic
field and $\pi$ is an unramified representation of $G_r$. We let $\phi^0$ be the characteristic 
function of $\of^n$. If $r$ is even
(resp. odd) we can choose $W^{0} \in \wpif$ and $\phi^{0} \in \sfnf$ 
(resp. $W^{0} \in \wpif$) such that 
\[
J(s,W^0,\phi^0)=\lepif\quad(\text{resp.}\quad J(s,W^0)=\lepif).
\]
\end{proposition}

\subsection{The global integrals} \label{irmglobal}
Let $k$ be a number field. Let $\Phi$ be a function in $ \cs(\ak^{n})$, 
the space of 
Schwartz-Bruhat functions on $\ak^n$. Let $\Pi=\otimes_v^{\prime}\Pi_v$ be 
a unitary cuspidal automorphic representation 
of $G_r(\ak)$ and $\omega_{\pi}$ be the central character of $\pi$. 
We remind the reader that the notation $[H]_k$ denotes the quotient
${\rm{H}}(k)\backslash {\rm{H}}(\ak)$ for any algebraic group defined
over a number field $k$. For a 
non-trivial 
additive character $\psi$ of $\ak/k$ and a form $\varphi$ in the space 
of $\Pi$, we let 
\begin{equation*}
 W_{\varphi}(g) = \int_{[N_r]_k} \varphi(ug)\psi(u)du 
\end{equation*}
be the Whittaker function associated to $\varphi$, where, as is usual,
 we view $\psi$ as a character of ${\rm{N}}_r(\ak)$ by setting
$\psi(u) = \displaystyle{\prod_{j=1}^{r-1}}\psi(u_{j,j+1})$.
If $r = 2n$, we consider the global integral
\begin{align*}
 J(s, W_{\varphi}, \Phi) = \int_{{\rm{N}}_n(\ak)\backslash {\rm{G}}_n(\ak)}\int_{\liebn(\ak)\backslash \matn(\ak)} &
 W_{\varphi}\left(\sigma\left(\begin{array}{cc} I_{n} & X \\ 0 & I_{n}
\end{array}\right)
\left(\begin{array}{cc} g & 0 \\ 0 & g \end{array}\right) \right)
\hspace{2cm}\\
&\psi(\tr X)dX \Phi(e_{n}g) |{\det g}|^{s} dg,
\end{align*}
where $ \sigma $ is the permutation given by
\begin{equation*}
 \sigma = \left( \begin{array}{ccccccccc} 
1 & 2 & \cdots & n & | & n+1 & n+2 & \cdots & 2n \\
1 & 3 & \cdots & 2n-1 & | & 2 & 4 & \cdots & 2n 
\end{array} \right).
\end{equation*}
If $r = 2n-1 $, consider 
\begin{align*} 
J(s,W_{\varphi})= 
\int_{{\rm{N}}_n(\ak)\backslash {\rm{G}}_n(\ak)}
\int_{\liebn(\ak)\backslash \matn(\ak)} & W_{\varphi}\left(\sigma
\left(\begin{array}{ccc} I_{n} & X & 0\\ 0 & I_{n} & 0\\ 0 & 0 & 1
\end{array}\right)
\left(\begin{array}{ccc} g & 0 & 0 \\ 0 & g & 0\\ 0 & 0 & 1 \end{array}
\right) \right) \\
 &\psi(\tr X)dX |{\det g}|^{s-1} dg, 
\end{align*}
where  $ \sigma $ is the permutation given by
\begin{center}
$ \sigma = \left( \begin{array}{cccccccccc} 
1 & 2 & \cdots & n & | & n+1 & n+2 & \cdots & 2n & 2n-1 \\
1 & 3 & \cdots & 2n-1 & | & 2 & 4 & \cdots & 2n & 2n-1
\end{array} \right) $.
\end{center}
For decomposable vectors $W_{\varphi}= \displaystyle{\prod_{v}}^{\prime} W_{v}$ and $\Phi = \displaystyle{\prod_{v}}^{\prime}\Phi_{v} $, 
where $v$ runs over all the absolute values of $k$, 
we have
\begin{equation*}
 J(s, W_{\varphi}, \Phi) = \prod_{v}J(s, W_{v}, \Phi_{v}) ~\text{and}~ J(s, W_{\varphi}) = \prod_{v}J(s, W_{v})
\end{equation*}
for ${\rm Re}(s)\gg0$. As for the $L$-functions, if $S$ is a set of places of $k$, it will be useful to introduce the notation
\[
J_S(s, W_{\varphi}, \Phi)=\prod_{v\in S}J(s, W_{v}, \Phi_{v}) ~\text{and}~ J_S(s, W_{\varphi}) = \prod_{v\in S}J(s, W_{v})
\]
As a consequence of Proposition 1 of Section 9 of \cite{JaSh90a} we have the following global functional equation.
\begin{theorem}\label{jsfnaleeqnthm}
Let $\Pi=\otimes_v^{\prime}\Pi_v$ be a unitary cuspidal automorphic representation of $G_r(\ak)$.
\begin{enumerate}
\item Let $r=2n$. For every $W_{\varphi}=\prod_v^{\prime}W_v\in \wPi$ and $\Phi=\prod_{v}^{\prime}\Phi_{v}\in \cs(\ak^{n})$,
\begin{equation}\label{eveglobintfneqn}
J(1-s,\rho(w_{n,n})\widetilde{W}_{\varphi},\widehat{\Phi})=J(s,W_{\varphi},\Phi)
\end{equation}
\item Let $r=2n-1$. For every $W_{\varphi}=\prod_v^{\prime}W_v\in \wPi$,
\begin{equation}\label{oddglobintfneqn}
J(1-s, \widetilde{W}_{\varphi,1})=J(s,W_{\varphi,2})).
\end{equation}
\end{enumerate}
\end{theorem}
At the finite unramified places $v$ of $\Pi$ we invoke Proposition \ref{unramvector} to choose the local data so that
$J(s,W_v,\Phi_v)=\lePiv$ when $r$ is even (resp. $J(s,W_{2,v})=\lePiv$ when $r$ is odd). Further, we can and will
assume that the local characters $\psi_v$ have been chosen so that $\epjsepiv=1$ when $v$ is a finite
unramified place. We let the data at the set of finite ramified places $S_r$ of $\Pi$ and the set of infinite places 
$S_{\infty}$ of $k$ be arbitrary, and set $S_{\infty,r}=S_{\infty}\cup S_r$. Let $S_f$ denote the set of all the finite places.
Using the local functional equations \eqref{padicevelocfneqn} and  \eqref{padicoddlocfneqn} at the finite ramified 
places, we can transform the equations \eqref{eveglobintfneqn} and \eqref{oddglobintfneqn} to the forms
\begin{equation}\label{intermedevefneqn}
J_{S_{\infty}}(s,W_{\varphi},\Phi_v)L_{S_f}(s,\Pi,\wedge^2)=
\varepsilon_{JS,S_r}(s,\Pi,\wedge^2,\psi)J_{S_{\infty}}(1-s,\widetilde{W}_{\varphi},\widehat{\Phi}_v)
L_{S_f}(1-s,\tilde{\Pi},\wedge^2)
\end{equation}
and
\begin{equation}\label{intermedoddfneqn}
J_{S_{\infty}}(s,W_{2,v})L_{S_f}(s,\Pi,\wedge^2)=\varepsilon_{JS,S_r}(s,\Pi,\wedge^2,\psi)J_{S_{\infty}}(1-s,\widetilde{W}_{1,v})
L_{S_f}(1-s,\tilde{\Pi},\wedge^2)
\end{equation}
respectively, according to whether $r$ is even or odd. Using the equations \eqref{padicevelocfneqn} and \eqref{padicoddlocfneqn} we see that the equations 
\eqref{intermedevefneqn} and \eqref{intermedoddfneqn} can be transformed to 
\begin{equation}\label{jslfnaleqn}
L(s,\Pi,\wedge^2)=\varepsilon_{JS,S_{\infty}}(s,\Pi,\wedge^2,\psi)\varepsilon_{JS,S_r}(s,\Pi,\wedge^2,\psi)
L(1-s,\tilde{\Pi},\wedge^2).
\end{equation}
In view of the factorisation of the global $\varepsilon$-factor into the local $\varepsilon$-factors, we can rewrite
the functional equation \eqref{lshfnaleqn}. We record this below as
\begin{equation}\label{lshfnaleqntwo}
L(s,\Pi,\wedge^2)=\varepsilon_{Sh,S_{\infty}}(s,\Pi,\wedge^2,\psi)\varepsilon_{Sh,S_r}(s,\Pi,\wedge^2,\psi).
L(1-s,\tilde{\Pi},\wedge^2).
\end{equation}

\section{Analyticity of the $L$-factors, $\varepsilon$-factors and local Zeta integrals}\label{secanalyticity}

As before, let $\pi_{\lambda}=\Ind_P^{G_r}(\sigma_{\lambda})$ be a generic representation of $G_r$ and let
$F$ be a $p$-adic field. In this section we study the behaviour of the $L$-factors, $\varepsilon$-factors 
and local Zeta integrals in the parameter $\lambda$.

\begin{proposition}\label{ntlfnisholfinord} Suppose
$\pi_{\lambda}\in \grepfr$, $\lambda \in ({\mathcal A}_{M,\C})^{*}$. For $*=G,Sh$, we have
\begin{enumerate}
 \item Let $U=\{\lambda=(\lambda_1,\lambda_2,\ldots ,\lambda_k)\in ({\mathcal A}_{M,\C})^{*} \,\big\vert\vert\,
\re{\lambda_i}\vert<1/4\}$. 
For every $s\in\C$ with $\res =1/2$, the functions
\begin{equation*}
\lambda \mapsto \varepsilon_*(s,\pi_{\lambda},\wedge^2,\psi)\quad\text{and}\quad
\lambda \mapsto \gamma_*(s,\pi_{\lambda},\wedge^2,\psi)
\end{equation*}
are holomorphic in $U$.

\item For any $s\in\C$, the function $(s,\lambda)\to L(s,\pi_{\lambda},\wedge^2)^{-1}$ is entire and of finite order in vertical strips in the first variable locally uniformly in the second variable. Moreover, if $\pi$ is nearly
tempered the function is nonvanishing in $\uh_{1/2-\epsilon}$ for some $\epsilon>0$.
\end{enumerate}
\end{proposition}
The proof of the first part of the proposition above is easy, and amost identical to the proof in the archimedean 
case given in \cite{JLST2025}  (and similar to the proof of Lemma 3.2.1 of \cite{BP2021} which deals 
with the Asai factors). The second part is almost trivial in our $p$-adic setting.

In the remainder of this section we recall two results from Section 3 of \cite{JLST2025}.
The proposition below establishes the convergence of the archimedean local integrals
in a suitable half-plane by following the argument in \cite{JaSh90a}, but 
makes the dependence on the parameter $\mu$ explicit using Lemma 2.5.1 of \cite{BP2021}. It 
is the direct analogue of Lemma 3.3.1 of \cite{BP2021}  which proves a similar result for the Flicker 
integrals for the Asai factors. This appears as Proposition 3.6 in Section 3 of \cite{JLST2025} 
(in mildly different notation from ours).

\begin{proposition}\label{jsintlambdahol}
Let $\mu\in \mathcal{A}^*$ and let
$\min(\mu):=\min_{1\le j\le r}\{\mu_j\}$ for $r=2n$ (resp. $r=2n-1$). Then for every   $W\in \mathcal{C}_{\mu}(N_r(F)\backslash G_r(F),\psi)$ and $\phi\in \mathcal{S}(F^n)$ (resp. every ($W\in \mathcal{C}_{\mu}(N_r(F)\backslash G_r(F),\psi)$), the integral $J(s,W,\phi)$ (resp. $\jswf$)
is absolutely convergent for all 
$s\in \uh_{-2\operatorname{min}(\mu)}$. Moreover, the function $s\mapsto J(s,W,\phi)$ (resp. $\jsw$) is 
holomorphic and bounded in vertical strips in this domain. For every vertical strip $V\subseteq \uh_{>-2\min(\mu)}$ there exists a continuous semi-norm 
$p_{V,\phi}$ on $\mathcal{C}_{\mu}(N_r(F)\backslash G_r(F),\psi)$ such that
\[
 \left\lvert J(s,W,\phi)\right\rvert\leqslant p_{V,\phi}(W)
 \]
for every $W\in \mathcal{C}_{\mu}(N_r(F)\backslash G_r(F),\psi)$ and $s\in V$. In particular,
the map $(W,\phi)\mapsto J(s,W,\phi)$ is continuous.
\end{proposition}

If $\pi=\pi_{\lambda}$ is nearly tempered, $|\re{\lambda_i}|<1/4$, $1\le i\le 1/4$, and we have the
following corollary from \cite{JLST2025}.
\begin{corollary}\label{ntjsintisholandbdd}
 Assume that $\pi\in \ntemprepfr$. Then, there exists 
 $\delta>0$ such that for every $W\in \wpiflambda$ and $\phi\in \sfnf$ (resp. $W\in \wpiflambda$)
 the integral $\jswphif$ (resp. $\jswf$) is holomorphic and bounded in 
 vertical strips in $\uh_{1/2-\epsilon}$.
\end{corollary}

\section{The equality of the $\varepsilon$-factors for supercuspidal representations}\label{secscrep}

We will first establish the equality $\epjsepif=\epshepif$ for supercuspidal representations $\pi$
using an old globalisation theorem of Henniart in Appendix I of \cite{Henniart1984}. 
As we mentioned in the introduction, a more recent reference is the relative version of the theorem proved
in \cite{PrSc2008} from which the theorem below can also be deduced.
\begin{theorem}\label{henniart84} Let $F$ be a $p$-adic field and let $\pi$ be a supercuspidal representation 
of $G_r$. There is a number field $k$, a place $v_0$ of $k$ such that $k_{v_0}\cong F$, and a cuspidal 
automorphic representation 
$\Pi=\otimes_v^{\prime}\Pi_v$ of $\glr(\ak)$ such that $\Pi_{v_0}\cong \pi$ and $\Pi_v$ is unramified for 
all finite places $v\ne v_0$.
\end{theorem}
This allows us to prove the following proposition which is a special case of our main theorem almost immediately.
\begin{proposition}\label{scepjsequalsepshprop} Let $\pi$ be a supercuspidal representation of $G_r$. Then
\begin{equation}\label{scepjsequalsepsheqn}
\epjsepif=\epshepif.
\end{equation}
\end{proposition}
\begin{proof} Let $k$ be a number field and $\Pi=\otimes_v^{\prime}\Pi_v$ be a cuspidal 
automorphic representation of $\glr(\ak)$ as in the theorem. For any cuspidal automorphic 
representation $\Pi$ of $\glr(\ak)$, where  $k$ is a number field, dividing the global functional 
equation \eqref{jslfnaleqn} by the global functional equation \eqref{lshfnaleqntwo} yields
\[
\varepsilon_{JS,S_r}(s,\Pi,\wedge^2,\psi)=\varepsilon_{Sh,S_r}(s,\Pi,\wedge^2,\psi),
\]
since Theorem \ref{archallepequaleqn} shows that 
$\varepsilon_{JS,S_{\infty}}(s,\Pi,\wedge^2,\psi)=\varepsilon_{Sh,S_{\infty}}(s,\Pi,\wedge^2,\psi)$.
We now choose $k$ and $\Pi$ as in Theorem \ref{henniart84}. In that case the set $S_r=\{v_0\}$, so 
the result follows immediately.
\end{proof}

\begin{remark} Recall that establishng the equality $\epshepif=\epgepif$ for supercuspidal representations, 
or rather, the even weaker equality obtained after twisting by highly ramified characters was the single most
difficult part of the proof of the main theorem \cite{CoShTs2017} and was the focus of most of that 
paper. Once this was established the passage to all generic representations was easier and relied primarily on the inductive 
formul{\ae} given in \eqref{epfnofparind}, and certain globalisation arguments of Henniart. In our case, the analogue of
\eqref{epfnofparind} for $\epjsepif$ was not known before the results of this paper were proved, so we could not use them to
reduce the generic case to the supercuspidal one.
\end{remark}

\section{Global methods for generic $p$-adic representations}\label{globmethpadic}

In this section we will complete the proof of Theorem \ref{epjsequalepshthm}, while also giving new (and global) proofs
of Theorems \ref{xientirethm} and \ref{padiclocfneqnthm} along the way. Our starting point will be the following  consequence
of a theorem of Shin in \cite{Shin2012}.
\begin{theorem}\label{shinglobthmtwo} Let $k$ be a number field, and let $v_0$ be a finite place of $k$.
Let $V$ be any open subset of $\operatorname{Temp}(\glr(k_{v_0}))$. There exists a cuspidal 
automorphic representation $\Pi=\otimes_v^{\prime}\Pi_v$ of $\glr(\ak)$ 
such that $\Pi_{v_0}\in V$ and $\Pi_v$ is either unramified or supercuspidal for every 
finite place $v\ne v_0$. 
\end{theorem}
\begin{proof} Our theorem is follows easily by making appropriate choices of functions in 
Theorem 4.8 of \cite{Shin2012} to which we refer for the notation of this proof. Recall that the hypotheses 
of the theorem require a choice of a finite set $S$ of 
places, two finite places $v_1,v_2\notin S$, a sequence of functions $\phi_n^{S,v_1,v_2}$, 
a fixed function $\phi_S$, and a fixed function
$\phi_{v_1,v_2}$ satisfying certain properties. Let $S_{\infty}$ denote the set of infinite places,
and let $S=S_{\infty}\cup \{v_0\}$. We choose $\phi_{\infty}$ on $\glr(k_{S_{\infty}})$ such that $\widehat{\phi_{\infty}}$ 
is not identically zero and $\phi_{v_0}$ such that $\widehat{\phi_{v_0}}$ is not identically zero in $V$. 
Let $\phi_S=\phi_{\infty}\phi_{v_0}$. Let $v_3\neq v_0,v_1,v_2$ be a finite place.
Choose 
$\phi_{v}$ to be the characteristic function of the maximal compact subgroup $K_{r,v}$ for all $v\ne v_1,v_2,v_3$ when $v\nmid\infty$. 
Choose $\phi_{v_1}$ to be a ``truncated Kottwitz function" of a supercuspidal representation $\pi_1$ of $\glr(k_{v_1})$. 
In this case it will be just a normalised matrix coefficient of $\pi_1$ (see Lemma 12 of \cite{Clozel1986}).
Choose $\phi_{v_2}$ to be a (normalised) matrix coefficient of a supercuspidal representation 
$\pi_2$ of $\glr(k_{v_2})$ with $\widehat{\phi_{v_2}}$ positive on a single supercuspidal Bernstein component and vanishing outside 
of it. At a third place $v_3\notin S_{\infty}, v_1,v_2$ we choose 
$\phi_{n,v_3}$ to a be sequence of (normalised) matrix coefficients of supercuspidal representations 
$\Pi_{n,3}$ of $\glr(k_{v_3})$ of increasing levels so that their supports shrink to $1$. If we set 
$\phi_n^{S,v_1,v_2}=\prod_{v\ne v_1,v_2,v\notin S} \phi_v$ and 
$\phi_{v_1,v_2}=\phi_{v_1}\phi_{v_2}$, then the choice $\phi_n=\phi_n^{S,v_1,v_2}\phi_{v_1,v_2}\phi_S$ in 
Theorem 4.8 of \cite{Shin2012} produces the desired global representation for $n$ large enough -- it will be supercuspidal at 
$v_1$, $v_2$ and $v_3$, and unramified at all other finite places for some $\Pi_{v_0}\in V$.
\end{proof}

\begin{remark} We can also take $v_0$ to be an infinite place and obtain the same conclusion.
We further suspect (though we have not checked this)
that we can find a global representation which is supercuspidal at a single finite place and unramified at all 
other finite places while $\Pi_{v_0}\in V$. 
However, we have chosen to retain all of the structure of 
Shin's proof and notation, and this yields a representation with three supercuspidal components without 
any additional arguments.
\end{remark}
\begin{remark} Note that in contrast to the rest of the paper Theorem 4.8 of \cite{Shin2012} does not require the 
group $G$ to have a discrete series representation. Likewise, the requirement that the set $S$ in the theorem 
contains all infinite places is absent in the other sections of the paper.
\end{remark}

\begin{proposition} \label{denseeppadic} Let $F$ be a $p$-adic field, and let $V$ be any open subset of 
$\operatorname{Temp}(G_r)$. There exists $\pi\in V$ such that 
\begin{equation}\label{eparchjseqsh}
\epjsepif=\epshepif.
\end{equation}
In particular, \eqref{epjsequalepsheqn} holds for a dense subset of $\operatorname{Temp}(G_r)$.
\end{proposition}
\begin{proof} 
We retain the notation of Theorem \ref{shinglobthmtwo}. Choose a number field $k$ and a place $v_0$ of $k$ 
such that $k_{{v_0}}\cong F$.
Fix a neighbourhood $V$ of $\operatorname{Temp}(\gln(k_{v_0}))$,
and choose finite places $v_i\mid p_i$, and such that $p_i$ splits in $k$, $i=1,2,3$. This means that
$k_{v_i}\cong \Q_{p_i}$. Choose $\Pi=\otimes_v^{\prime}\Pi_v$ as in Theorem \ref{shinglobthmtwo}, so $\Pi_{\infty}\in V$. 
Note that the set $S=\{v_i\}_{i=1}^3$ coincides with the set $S_r$ of places where $\Pi_v$ is ramified. As before,
we compare the global functional equations \eqref{jslfnaleqn} and \eqref{lshfnaleqntwo}. Since the 
$L$- and $\varepsilon$-factors coincide at the infinite places (by Theorem \ref{archepequalthm})
, unramified places and supercuspidal places (by Theorem \ref{scepjsequalsepsheqn}), we see
that
\[
\varepsilon_{JS}(s,\Pi_{v_0},\wedge^2,\psi)=\varepsilon_{Sh}(s,\Pi_{v_0},\wedge^2,\psi).
\]
This proves our proposition.
\end{proof}

Our strategy going forward is identical to the one employed for the Asai $\varepsilon$-factor in \cite{BP2021}, and
indeed, we will follow the notation of that paper very closely. It is also, in part, the approach used to complete the archimedean theory in \cite{JLST2025}. We start by giving another the proof of the local functional equation for nearly tempered representations when $\re{s}=1/2$.
\begin{proposition}\label{feforntemp} Let $F$ be a $p$-adic field and let $\pi\in \ntemprepfr$. If $r=2n$, for every 
$W\in \wpif$, $\phi\in \sfnf$, and $s\in \C$ with $\re{s}=1/2$,
\begin{equation}\label{evenntempfneq}
\widetilde{\Xi}(1-s,W, \phi)
= \varepsilon_{Sh}(s,\Pi,\wedge^2,\psi)\Xi(s,W,\phi),
\end{equation}
while if $r=2n-1$, for every $W\in \wpif$, 
\begin{equation}\label{oddntempfneq}
\widetilde{\Xi}(1-s,W_{1})
=\varepsilon_{Sh}(s,\Pi,\wedge^2,\psi) \Xi(s,W_{2}).
\end{equation}
\end{proposition}
\begin{proof}
As usual, we write $\pi=\Ind_{P}^{G_r}(\sigma_{{\lambda}_0})$, where $P=MU$ is a standard parabolic
subgroup of $G_r$, $\sigma\in \dsrepfh(M)$ and $\lambda_0=(\lambda_{0,1},\ldots ,\lambda_{0,k})\in 
{\mathcal{A}}_{M,\C}^{*}=\C^k$ such that $\re{|\lambda_{0,i}|}< 1/4$, $1\le i\le k$. We can assume (after twisting
by a central character if necessary) without loss of generality that 
$\lambda_0\in ({\mathcal{A}}_{M,\C}^{G_r})^{*}$. We set $\pi_{\lambda}=\Ind_P^{G_r}(\sigma_{\lambda})$ for 
all $\lambda\in {\mathcal{A}}_{M,\C}^{*}$. Let $\rho$ be half the sum of the positive roots of $T_r$ in $N_r$.
Because $\pi$ is nearly tempered, we can choose $\epsilon>0$ such that
$\min(|\re{\lambda_0}|+\epsilon\rho)>-1/4$. Define $\mu=|\re{\lambda_0}|+\epsilon\rho$ and consider the open set ${\mathcal{U}}[\prec\mu]$. 
Let $W_{0}\in \wpif$ and $\phi_{0}\in \sfnf$. Proposition \ref{whittanalytic} tells us that there is an analytic map
$\lambda\mapsto W_{\lambda}$ for every $\lambda\in {\mathcal{U}}[\prec\mu]$.

Let $r=2n$. By Proposition \ref{denseeppadic}, we know that there is a dense subset $D$ of 
$\operatorname{Temp}(G_r)$ such that equation \eqref{evenntempfneq} holds. If the functional equation 
holds for a given $\pi$, it obviously holds for $\pi$ twisted by an unramified central character. It follows that 
there is a dense subset $D_{\sigma}$ of $i({\mathcal{A}}_{M}^G)^{*}$ such that \eqref{evenntempfneq} holds
for all $\lambda\in D_{\sigma}$ and all $s\in \C$ with $\re{s}=1/2$. Further, by Proposition \ref{jsintlambdahol} and
Corollary \ref{ntjsintisholandbdd}, we know that for $\re{s}=1/2$ the maps $\lambda\to J(s,W_{\lambda},\phi_0),J(1-s,\rho_{n,n}\tilde{W}_{\lambda},\widehat{\phi_0})$ are holomorphic functions for $\lambda\in {\mathcal{U}}[\prec\mu]$, while
Proposition \ref{ntlfnisholfinord} guarantees that $\lambda\mapsto \gamma_{Sh}(s,\pi_{\lambda},\wedge^2,\psi)$
is also a holomorphic map in the same domain. By continuity established in \ref{jsintlambdahol}, it follows 
that \eqref{evenntempfneq} holds for all 
$\lambda\in i({\mathcal{A}}_{M}^G)^{*}$. By the holomorphy in $\lambda$, \eqref{evenntempfneq} holds for all 
$\lambda\in {\mathcal{U}}[\prec\mu]$. Thus, \eqref{evenntempfneq} must hold for $\lambda=\lambda_0$,
proving the proposition when $r=2n$.

The proof when $r=2n-1$ is identical. 
\end{proof}

Since we have established the local functional equation in the nearly tempered case, we can establish
Theorem \ref{xientirethm} in this case almost immediately. There is no difference between the odd and 
even cases, so we give the argument for $r=2n$.
By Corollary \ref{ntjsintisholandbdd},
we know that if $\pi$ is nearly tempered, 
$\epshepif\Xi(s,W,\phi)$ is holomorphic in $\re{s}>1/2-\epsilon$ for some 
$\epsilon$, while $\tilde{\Xi}(1-s,\tilde{W},\hat{\phi})$ is holomorphic
in $\re{s}<1/2+\epsilon$. Equation \eqref{evenntempfneq} shows that $\Xi(s,W,\phi)$ is entire since
$\epjsepif$ is non-vanishing on all of $\C$. Moreover, by Propositions \ref{ntlfnisholfinord}
and \ref{jsintlambdahol}, we know that $\Xi(s,W,\phi)$ is of finite order in vertical strips. Further, for each $s_0\in \C$, 
Proposition \ref{archnonvanish} guarantees the existence of a $W\in \wpif$ and a $\phi\in \sfnf$ such that $J(s_0,W,\phi)\ne 0$ (this is true for any generic representation $\pi$, not just for the nearly tempered ones). We summarise our discussion as
\begin{proposition}\label{xiholforntemp} If $\pi\in \ntemprepfr$, $\Xi(s,W,\phi)$ is entire and of finite order in vertical strips for all
$W\in \wpif$ and all $\phi\in \sfnf$. Further, for each $s_0\in \C$, there exists $W\in \wpif$ such that $\Xi(s_0,W,\phi)\ne 0$.
\end{proposition}
\begin{remark} Obviously the proposition above also holds for $\tilde{\Xi}(1-s,W,\phi)$.
\end{remark}

We can now complete our proof of Theorem \ref{padiclocfneqnthm}. Again, the arguments are similar
to those employed for archimedean representations in \cite{JLST2025}, and in Section 3.10 of 
\cite{BP2021} for the Asai $L$-function. 
\begin{proof}[Proofs of Theorem \ref{xientirethm} and Theorem  \ref{padiclocfneqnthm}]
Let $\pi\in \grepfr$. As usual write $\pi_{\lambda_0}=\Ind_P^{G_r}\sigma_{\lambda_0}$ for a suitable standard parabolic $P=MU$,
$\sigma\in \dsrepfh(M)$ and $\lambda_0\in {\mathcal{A}}_{M,\C}^{*}$. As before, after twisting by an unramified central character
we can assume that $\lambda_0\in ({\mathcal{A}}_{M,\C}^{G_r})^{*}$. Let $W\in \wpif$.
Recall that for $\mu\in ({\mathcal{A}}_{M}^{G_r})^{*}$, we can choose a section $\lambda\mapsto W_{\lambda}$ which is 
analytic in ${\mathcal{U}}[\mu]$ and such that $W_{\lambda_0}=W$.

As usual we will focus on the case $r=2n$, the case $r=2n-1$ being similar. Consider the functions
\begin{flalign*}
Z_{+}(s,\lambda)&=\varepsilon_{Sh}(s+1/2,\pi_{\lambda},\psi)\Xi(s+1/2,W_{\lambda},\phi),\,\,\text{and}\\
Z_{-}(s,\lambda)&=\varepsilon_{Sh}(s+1/2,\pi_{\lambda},\psi)\tilde{\Xi}(s+1/2,W_{\lambda},\phi).
\end{flalign*}
Let $U=\{\lambda\in ({\mathcal{A}}_{M,\C}^{G_r})^{*}\,\vert\,|\re{\lambda_i}|<1/4, 1\le i\le k\}$, so $\pi_{\lambda}\in \ntemprepfr$ 
for every $\lambda\in U$. The functions $Z_{+}$ and $Z_{-}$ are holomorphic on $\uh_{\ge 1/2}\times U$ and, by Proposition 
\ref{xiholforntemp} admit holomorphic extensions to $\C\times U$ which are of finite order in vertical strips in the first variable locally 
uniformly in the second variable. Further, since \eqref{evenntempfneq} has been proved for $\pi\in \ntemprepfr$, we see that $Z_{+}(s,\lambda)=Z_{-}(-s,\lambda)$ for all $s\in \C$ and $\lambda\in U$. If $U_1$ is a relatively compact connected open subset of 
$({\mathcal{A}}_{M,\C}^{G_r})^{*}$, we may realise it as an open subset of ${\mathcal{U}}[\mu]$ for a suitable $\mu$.
Then $Z_{+}$ and $Z_{-}$ admit holomorphic extensions to $\uh_{>C}\times U_1$ for some $C$, and are of finite order in vertical 
strips in the first variable locally uniformly in the second variable. Now Proposition 2.8.1 of \cite{BP2021} asserts that both functions
extend to holomorphic functions on $\C\times ({\mathcal{A}}_{M,\C}^{G_r})^{*}$ which are of finite order in vertical 
strips in the first variable locally uniformly in the second variable and continue to satisfy the local functional functional equation on the
larger domain. In particular, this holds for $\lambda_0$. This proves Theorem \ref{xientirethm}.

Since we now have the local functional equation for all tempered representations and 
\[
\varepsilon_{JS}(s,\Pi_{v_0},\wedge^2,\psi_{v_0})=\varepsilon_{Sh}(s,\Pi_{v_0},\wedge^2,\psi_{v_0}).
\]
holds on a dense subset, the equality must hold everywhere in $\temprepfr$ by  continuity (Proposition \ref{jsintlambdahol}). By the holomorphy in $\lambda$,
the equality must hold for all generic representations. This proves Theorem \ref{epjsequalepshthm}.

Further, since $\epshepif$ is known to have the form $A(\pi)B(\pi)^s$,  so must $\epjsepif$, which completes the proof of Theorem \ref{padiclocfneqnthm}.
\end{proof}

\begin{remark} Theorem 3.8.1 of \cite{BP2021} is a fairly precise globalisation result. Unfortunately, the 
globalisation produced by this theorem loses track of local data at one finite place -- in particular, we do not 
know if the representation is supercuspidal at this place, so it is inadequate for our purposes. 
\end{remark}

\bibliographystyle{alpha}
\bibliography{../../../../Bibtex/master2026}

\end{document}